%% file: NoGapTameArxiv.tex
\documentclass[11pt]{amsart}

\usepackage{bhit}

\title{The No Gap Conjecture for tame hereditary algebras}

\input{AuthorInformation} %% Include Author Info 

\begin{document}

\begin{abstract}
The ``No Gap Conjecture'' of Br\"ustle-Dupont-P\'erotin states that the set of lengths of maximal green sequences for hereditary algebras over an algebraically closed field has no gaps. This follows from a stronger conjecture that any two maximal green sequences can be ``polygonally deformed'' into each other. We prove this stronger conjecture for all tame hereditary algebras over any field. 
\end{abstract}

\maketitle

%%%%%%%%%%%%%%%%%%%%%%%%%%%%%%%%%%%%%%%%%%
%%%
%%%   Introduction
%%%
%%%%%%%%%%%%%%%%%%%%%%%%%%%%%%%%%%%%%%%%%%

\section*{Introduction}

For any finite dimensional hereditary algebra over any field (and more generally for any cluster tilted algebra \cite{BMR}) there is a notion of a ``maximal green sequence'' which has gained considerable attention in recent years due to its relation to many areas of algebraic combinatorics (\cite{BDP}, \cite{Keller}) and theoretical physics (\cite{ACCERV}, \cite{X}). For example, the maximal chains in the Tamari lattice \cite{Tamari} are in bijection with maximal green sequences for the quiver of type $A_n$ with straight orientation. 

Many of the open problems about maximal green sequences have recently been solved. One of the main conjectures is that there are only finitely many maximal green sequences. This was proved for all tame hereditary algebras in \cite{BDP} and extended to cluster tilted algebras of tame type in \cite{BHIT}. Although every hereditary algebra admits a maximal green sequence of length equal to its rank (the number of nonisomorphic simple modules which is also the number of vertices in the associated quiver), it was recently shown that not all cluster tilted algebras have maximal green sequences \cite{M}.

This paper solves, in the tame case, another conjecture about maximal green sequences known as the ``No Gap Conjecture''. This is the conjecture of Br\"ustle, Dupont and P\'erotin \cite{BDP} which states that, for hereditary algebras over an algebraically closed field, there is no gap in the sequence of lengths of maximal green sequences. This statement is not true for modulated quivers of types $B_2$ and $G_2$ since the tensor algebras of such quivers are hereditary algebras having only two maximal green sequences|of length $2$ and $4$ in type $B_2$, and lengths $2$ and $6$ in type $G_2$. Garver-McConville prove this conjecture in \cite{GM} for algebras which are cluster tilted of type $A_n$ and those whose quivers are cyclic. We begin this paper with the observation that the argument of Garver and McConville extends easily to all cluster tilted algebras of finite type over an algebraically closed field (Theorem \ref{thm: no gap for finite type} below). In fact there is a stronger statement which holds over any field. Namely, any two maximal green sequences are ``polygonally deformable'' into each other (Definition \ref{def: polygonal deformation}). In the simply laced case, this implies the No Gap Conjecture. A result of Reading \cite{R1X} states that this holds whenever the poset of functorially finite torsion classes in the category of finite dimensional modules $\rmod\Lambda$ forms a lattice. In particular, this applies to all path algebras of finite type by \cite{IRTT}.

In this paper we prove this stronger conjecture for all tame hereditary algebras over any field. Namely, we show that, for these algebras, any two maximal green sequences are polygonally deformable into each other. This implies the No Gap Conjecture for tame path algebras. The proof uses basic idea from \cite{GM} and \cite{R1X} and extends it to the tame case using results from \cite{BHIT}.

%%%%%%%%%%%%%%%%%%%%%%%%%%%%%%%%%%%%%%%%%%
%%%
%%%   Basic Definitions
%%%
%%%%%%%%%%%%%%%%%%%%%%%%%%%%%%%%%%%%%%%%%%

\section{Basic Definitions}

We give the basic definition of a maximal green sequence from the point of view of representation theory and combinatorics since we will use both languages. We also review the correspondence between these two languages from \cite{IOTW2}. 

\subsection{Representation theory}

Let $\Lambda$ be a finite dimensional hereditary algebra over a field $K$. A (finitely generated right) $\Lambda$-module $M$ is called \emph{rigid} if $\Ext^1_\Lambda(M,M)=0$. It is called \emph{exceptional} if it is rigid and indecomposable. Two exceptional modules $T_1,T_2$ are called \emph{compatible} if $T_1\oplus T_2$ is rigid, i.e., $T_1,T_2$ do not extend each other. A \emph{tilting module} for $\Lambda$ is defined to be a rigid module with a maximum number of (nonisomorphic) indecomposable summands. It is well-known that this number $n$ is the number of nonisomorphic simple $\Lambda$-modules. This is equivalent to a choice of $n$ compatible exceptional modules. For example, $\Lambda=P_1\oplus\cdots\oplus P_n$ is a tilting module where $P_i$ denotes the projective $\Lambda$-module corresponding to the vertex $i$.

We define an \emph{exceptional object} to be either an exceptional modules or a \emph{shifted projective object} $P_i[1]=(P_i\to 0)$. These are all objects in the bounded derived category of $\Lambda$. Shifted projective objects are defined to be compatible with each other and with any module $M$ for which $\Hom_\Lambda(P_i,M)=0$. A \emph{cluster tilting object} for $\Lambda$ is defined to be a maximal rigid object in the derived category of $\Lambda$ whose components are exceptional objects. It follows from the definitions that a direct sum of exceptional objects is rigid if and only if its summands are pairwise compatible. One of the basic theorems of cluster theory is the following.

\begin{thm}[\cite{BMRRT}]
Every cluster tilting object for $\Lambda$ has $n$ summands $T=T_1\oplus\cdots\oplus T_n$. For any $1\le k\le n$ there is a unique exceptional object $T_k'$ not isomorphic to $T_k$ so that $T'=T/T_k\oplus T_k'$ is a cluster tilting object.
\end{thm}

We say that $T'=T/T_k\oplus T_k'$ is obtained from $T$ by \emph{mutation in the $k$-th direction} and we write $T'=\mu_kT$. The objects $T_k,T_k'$ are called an \emph{exchange pair}. Another fundamental theorem is the following. (Here we use the fact that the endomorphism ring $F_X=\End_{\cD^b}(X)$ of any exceptional object $X$ is a division algebra over $K$.)

\begin{thm}[\cite{BMRRT}]
Two exceptional objects $X,Y$ form an exchange pair if and only if either $\Ext^1_{\cD^b}(X,Y)$ or $\Ext^1_{\cD^b}(Y,X)$ is zero and the other is one-dimensional over both $F_X$ and $F_Y$.
\end{thm}

\emph{Green mutations} are defined in Definition \ref{defn: max green seq}. The following characterization of green mutations is from \cite{BDP}. It has the advantage of making it clear that mutations are either red or green. (Caution: ``red'' and ``green'' are reversed from the notation of \cite{BDP}.)

\begin{thm}[\cite{BDP}]
Let $T'=\mu_k T=T/T_k\oplus T_k'$. The mutation $\mu_k:T\mapsto T'$ is a green mutation (and $\mu_k: T'\to T$ is a red mutation) if and only if $\Ext^1_{\cD^b}(T_k',T_k)\neq0$.
\end{thm}

For example, any mutation of $\Lambda[1]=P_1[1]\oplus \cdots\oplus P_n[1]$ is green and no mutation of $\Lambda=P_1\oplus\cdots\oplus P_n$ is green.

\begin{defn}
The \emph{oriented exchange graph} $E(\Lambda)$ is defined to be the graph whose vertices are cluster tilting objects with an oriented edge $T\to T'$ when $T'$ is a green mutation of $T$. A \emph{maximal green sequence} is defined to be a maximal finite directed path in this graph.
\end{defn}

\begin{cor}
A sequence of green mutations of cluster tilting objects is maximal if and only if it starts at $\Lambda[1]$ and ends with $\Lambda$.
\end{cor}

\subsection{Combinatorics}

We recall the construction of the extended exchange matrix corresponding to a cluster tilting object \cite{BMRRT}, \cite{IOTW2}.

For $M,N$ finitely generated right $\Lambda$-modules the \emph{Euler-Ringel form} is defined by:
\[
	\brk{M,N}=\dim_K\Hom_\Lambda(M,N)-\dim_K\Ext^1_\Lambda(M,N).
\]
Let $S_1,\dots,S_n$ be the simple $\Lambda$-modules. For any $\Lambda$-module $M$, the \emph{dimension vector} of $M$ is defined to be the vector $\undim M=(a_1,\dots,a_n)\in\ZZ^n$ where $a_i$ is the number of times that $S_i$ occurs in the composition series of $M$. For a shifted projective object $P[1]$ we define $\undim P[1]:=-\undim P$. The \emph{Euler matrix} $E$ is the $n\times n$ integer matrix with entries $E_{ij}=\brk{S_i,S_j}$. The \emph{Euler-Ringel pairing} $\brk{\, \cdot\, ,\cdot\, }:\ZZ^n\times \ZZ^n\to\ZZ$ is defined on integer vectors by $
	\brk{x,y}=x^tEy
$. It is not too difficult to see that 
\[
	\brk{M,N}=\brk{\undim M,\undim N}
\]
for all finitely generated $\Lambda$-modules $M,N$.

By Schur's Lemma, $F_i=\End_\Lambda(S_i)$ is a division algebra over $K$ and each $E_{ij}$ is divisible by $f_i=\dim_K F_i$ and by $f_j$. This implies that $E$ can be factored as $E=LD=DR$ where $L,R$ are integer matrices and $D$ is the diagonal matrix with diagonal entries $f_i$. The \emph{initial exchange matrix} of $\Lambda$ is defined to be $B_0=L^t-R$. This matrix is \emph{skew-symmetrizable} since $DB_0$ is skew-symmetric. 

For any cluster tilting object $T$ for $\Lambda$ the corresponding extended exchange matrix $\widetilde B_T$ is given by the following explicit formula (cf. \cite{BMRRT}). We use the fact that $\det\,L=\det\,R=1$ since $L$ and $R$ are unipotent matrices.

\begin{thm}[$c$-Vector Theorem \cite{IOTW2}]
For any cluster tilting object $T=T_1\oplus\cdots\oplus T_n$, let $V$ be the $n\times n$ matrix whose columns are $\undim T_i$. Then
\begin{enumerate}
\item There is a unique $\Gamma\in GL(n,\ZZ)$ so that $V^tE\Gamma=D$.
\item The extended exchange matrix corresponding to $T$ is given by  
$
	\widetilde B_\Gamma=\mat{B_\Gamma\\-\Gamma}$ where $B_\Gamma=D^{-1}\Gamma^tDB_0\Gamma.
$
\end{enumerate}
\end{thm}

For example, if $T=\Lambda=P_1\oplus\cdots\oplus P_n$ then $\Gamma=I_n$, the identity matrix, since $E(P_i,S_j)=f_i\delta_{ij}$ and $B_{I_n}=B_0$. So, the corresponding extended exchange matrix is $\widetilde B_{I_n}=\mat{B_0\\-I_n}$. Similarly, the initial cluster tilting object $\Lambda[1]$ corresponds to the \emph{initial extended exchange matrix} $\widetilde B_0:=\widetilde B_{-I_n}=\mat{B_0\\I_n}$.

Note that $D\widetilde B_\Gamma=\Gamma^tDB_0\Gamma$ is skew-symmetric. The content of the $c$-Vector Theorem is that if $T$ is replaced with $\mu_kT$ then $\widetilde B_\Gamma$ is replaced by $\mu_k\widetilde B_\Gamma$ defined as follows.

\begin{defn}[\cite{FZ}]\label{def: mutation of B}
Let $\widetilde B=\mat{B\\  C}$ be any $2n\times n$ integer matrix where $DB$ is skew-symmetric and let $1\le k\le n$. Then the \emph{mutation of $\widetilde B$ in the $k$-th direction}, denoted $\mu_k \widetilde B$, is defined to be the matrix with entries
\[
	b_{ij}'=\begin{cases} -b_{ij} & \text{if } i=k \text{ or }j=k\\
b_{ij} +b_{ik}|b_{kj}| & \text{if } b_{ik}b_{kj}>0\\	
  b_{ij}  & \text{otherwise}
    \end{cases}
\]
\end{defn}

A (reachable) \emph{extended exchange matrix} $\widetilde B=\mat{B\\  C}$ is defined to be any matrix obtained by from the initial extended exchange matrix by iterated mutation. The columns of $C$ are called the \emph{$c$-vectors} of $\widetilde B$.

\begin{defn}[\cite{Keller}]\label{defn: max green seq}
A mutation $\mu_k$ is called \emph{green} if the entries of the $k$-th column of $C$ are all nonnegative. A \emph{maximal green sequence} is a sequence of green mutations starting with the initial extended exchange matrix $\widetilde B_0$ and ending in a matrix $\widetilde B$ whose bottom half $C$ has no positive entries.
\end{defn}

\begin{eg}\label{eg: C2}
Let $\Lambda=\mat{\CC & 0\\ \CC & \RR}$. This is an algebra of type $C_2$ with $4$ indecomposable modules. The Euler matrix is
\[
	E=\mat{2 & 0 \\ -2 & 1}=LD=\mat{1 & 0 \\ -1 & 1}\mat{2 & 0 \\ 0 & 1}=DR=\mat{2 & 0 \\ 0 & 1}\mat{1 & 0 \\ -2 & 1}.
\]
So, the initial exchange matrix is $B_0=L^t-R=\mat{0 & -1 \\ 2 & 0}$. There are exactly two maximal green sequences of length 2 and 4 respectively. (The first mutation is either $\mu_1$ or $\mu_2$. After that there is no choice but to alternate $\mu_1$ and $\mu_2$.)
\[
	\mat{0 & -1\\ 2 & 0\\ \hline 1 & 0\\ 0 & 1} \xrightarrow{\mu_1}
	\mat{0 & 1\\ -2 & 0\\ \hline -1 & 0\\ 0 & 1} \xrightarrow{\mu_2}
	\mat{0 & -1\\ 2 & 0\\ \hline -1 & 0\\ 0 & -1}
\]
\[
	\mat{0 & -1\\ 2 & 0\\ \hline 1 & 0\\ 0 & 1} \xrightarrow{\mu_2}
	\mat{0 & 1\\ -2 & 0\\ \hline 1 & 0\\ 2 & -1} \xrightarrow{\mu_1}
	\mat{0 & -1\\ 2 & 0\\ \hline -1 & 1\\ -2 & 1} \xrightarrow{\mu_2}
	\mat{0 & 1\\ -2 & 0\\ \hline 1 & -1\\ 0 & -1} \xrightarrow{\mu_1}
	\mat{0 & -1\\ 2 & 0\\ \hline -1 & 0\\ 0 & -1}
\]
We denote these maximal green sequences by $(1,2)$ and $(2,1,2,1)$ respectively.
\end{eg}

This calculation generalizes in the following well-known way \cite{FZII}.

\begin{prop}\label{prop: polygon for C2}
Suppose that $\widetilde B=\mat{B\\C}$ is an extended exchange matrix for which $\mu_j$ and $\mu_k$ are green mutations. Suppose $b_{jk}=-1$ and $b_{kj}=2$. Then the two maximal sequences of green mutations using only $\mu_j$ and $\mu_k$ are the sequences $(j,k)$ and $(k,j,k,j)$. Furthermore, $\mu_k\mu_j\widetilde B=\mu_j\mu_k\mu_j\mu_k \widetilde B$.
\end{prop}

\begin{proof}
Let $c_j,c_k$ be the $j$-th and $k$-th columns of $C$. These must be positive (have all coefficients $\ge0$). If we multiply $(c_j,c_k)$ on the right by the $2\times 2$ $c$-matrices in Example \ref{eg: C2} we will get the $j$-th and $k$-th columns of the bottom half of the corresponding mutations of $\widetilde B$. For example, the $j$-th and $k$-th columns of the $C$-matrix of $\mu_j\mu_k \widetilde B$ are $(c_j,c_k)\mat{-1&1\\-2&1}=(-c_j-2c_k,c_j+c_k)$.

To show that $\mu_k\mu_j\widetilde B=\mu_j\mu_k\mu_j\mu_k \widetilde B$ it suffices to examine the $c$-vectors $c_p$, i.e., the columns of the matrix $C$. If we let $b_{jp}=x,b_{kp}=y$ then, by the exchange formula \ref{def: mutation of B}, $c_p$ mutates to
\[
	c_p'=c_p+\max(0,x)c_j+\max(0,y,2x+y)c_k.
\]
Let $x'=\max(0,x),y'=\max(0,y,2x+y)$. For the mutation sequence $(k,j,k,j)$ we get the same vector $(x',y')\in\NN^n$ by a straightforward calculation, dividing into cases depending on the signs of $x,y,x+y,x+2y$ (cf. Reading \cite{R}).
The new $c$-matrix $C'$ determines $B'=(\mu_j\mu_k)^2B$ since, by a formula of \cite{NZ} which holds for any green mutation of any extended exchange matrix, $B'=D^{-1}X^tDBX$ where $X=C^{-1}C'$ is the matrix whose $j$-th and $k$-th rows we have computed (i.e., $X_{pj}=x'$, $X_{pk}=y'$, and $X$ differs from the identity matrix only in its $j$-th and $k$-th rows). Therefore, $\mu_k\mu_j B=(\mu_j\mu_k)^2B$.\end{proof}

%%%%%%%%%%%%%%%%%%%%%%%%%%%%%%%%%%%%%%%%%%
%%%
%%%   Polygonal Deformations and Regular Objects
%%%
%%%%%%%%%%%%%%%%%%%%%%%%%%%%%%%%%%%%%%%%%%

\section{Polygonal deformations and regular objects}

\subsection{Polygons} In forthcoming work, Nathan Reading defines a \emph{polygon} in a lattice $L$ to be a finite interval of the form $[x\vee y,x\wedge y]$ which is the union of two chains in $L$ which intersect only at the top and bottom (at $x\vee y,x\wedge y$). Since the poset of finite torsion classes for a tame algebra is not a lattice \cite{IRTT}, we use the following weaker notion which applies only to the oriented exchange graph $E(\Lambda)$ defined above.

\begin{defn}[Based on Reading \cite{R1X}]
A \emph{polygon} $P$ is a finite subgraph of the oriented exchange graph closed under multiple applications of two mutations, say $\mu_j,\mu_k$. 
\end{defn}

\begin{prop}\label{prop: possible polygons}
Suppose that $\widetilde B$ is an extended exchange matrix for which $\mu_j,\mu_k$ are both green mutations. Consider the longest sequences of green mutations of $\widetilde B$ starting with $\mu_j$ or $\mu_k$ and alternating between these. Then one sequence has length $2$ and the other has length $\ell=2,3,4,6$ or $\infty$ depending on whether $|b_{jk}b_{kj}|=0,1,2,3$ or $\ge4$. Furthermore, the results of these mutations on $\widetilde B$ give the same matrix up to permutation of the $j$-th and $k$-th rows and columns. Therefore, polygons in the oriented exchange graph have either $4,5,6$ or $8$ sides.
\end{prop}

\begin{proof}
(From \cite{FZII} and \cite{R}.) Up to sign, there are only four possible finite cases. Proposition \ref{prop: polygon for C2} proves one of the cases; the others are similar. The case $\ell=3$ is the only case where the results of the two mutation sequences ($\mu_k\mu_j\widetilde B$, $\mu_k\mu_j\mu_k\widetilde B$ in that case) are not equal.
\end{proof}

\begin{rem}\label{rem: 4 or 5}
The finite cases $\ell=2,3,4,6$ correspond to the rank 2 root systems $A_1\times A_1$, $A_2$, $B_2$, $G_2$ respectively. From this it follows that in the simply laced case, polygons can only have either 4 or 5 sides.
\end{rem}

\subsection{Polygonal deformations}

\begin{defn}[Based on Reading \cite{R1X}]\label{def: polygonal deformation}
A \emph{polygonal deformation} between two maximal green sequences is a sequence of local deformations replacing one side of a polygon with the other in the middle of a maximal green sequence. These local deformations will be called \emph{elementary polygonal deformations}.
\end{defn}

The No Gap Conjecture would follow from the statement that all maximal green sequences lie in the same polygonal deformation class since, in the simply laced case, an elementary polygonal deformation changes the length of a maximal green sequence by $0$ or $1$ (Remark \ref{rem: 4 or 5}).

Let $T,T'$ be cluster tilting objects of $\Lambda$ and let $\cM G(T,T')$ be the set of all green sequences from $T$ to $T'$. The set of objects in these sequences will be denoted $[T,T']$. Then
	\[
	[T,T']=\{T''\,:\, T\le T''\le T'\}
	\]
where $T\le T''$ if and only if $\cM G(T,T'')$ is nonempty.

\begin{lem}[Well known, cf. e.g. \cite{BHIT}]
A green sequence cannot mutate the same $c$-vector twice. In particular, a green sequence cannot come back to the same point $T$.
\end{lem}

\begin{lem}\label{lem: main lemma}
Suppose that $T$ is a cluster tilting object of $\Lambda$ with the property that there are only finitely many $T'\ge T$. Then $\cM G(T,\Lambda)$ is finite, nonempty, and any two elements of it lie in the same polygonal deformation class.
\end{lem}

\begin{proof}
First note that every green mutation sequence starting at $T$ will terminate in a finite number of steps to $\Lambda$ since $\Lambda$ is the only cluster tilting object with all negative $c$-vectors. In particular, $\cM G(T,\Lambda)$ is nonempty. Finiteness of $\cM G(T,\Lambda)$ is clear since green sequences cannot go through the same cluster twice.

Write $C\sim C'$ if $C,C'\in \cM G(T,\Lambda)$ lie in the same polygonal deformation class. To prove that all elements of $\cM G(T,\Lambda)$ are equivalent, choose $C_1\not\sim C_2$ so that the two mutation sequences $C_1,C_2$ from $T$ to $\Lambda$ have the largest possible initial common sequence, say $T=T^{(1)},T^{(2)},\dots,T^{(m)}$. The first sequence $C_1$ mutates $T^{(m)}$ to $\mu_j T^{(m)}$, the second $C_2$ mutates $T^{(m)}$ to $\mu_k T^{(m)}$ where $j\neq k$. These two mutations when alternated give a polygon since there are no infinite green sequences starting at $T$. The polygon goes from $T^{(m)}$ up to some $T'$. 

Since $T'\ge T^{(m)}\ge T$, there are at most finitely many $T''\ge T'$. Therefore, any green sequence starting with $T'$ will reach $\Lambda$ in a finite number of steps. Choose one such sequence. Together with the initial segment $T^{(1)},T^{(2)},\dots,T^{(m)}$ and the two sides of the polygon, this gives two more green sequences $C_3,C_4$ from $T$ to $\Lambda$ which differ by a polygonal deformation (i.e. $C_3\sim C_4$); moreover the initial $m+1$ segment of $C_3$ agrees with the initial $m+1$ segment of $C_1$ and the initial $m+1$ segment of $C_4$ agrees with the initial $m+1$ segment of $C_2$. Therefore by maximality of $m$, we have that $C_1\sim C_3$ and $C_2\sim C_4$. So, $C_1\sim C_3\sim C_4\sim C_2$ showing that $C_1\sim C_2$ for any two elements of $\cM G(T,\Lambda)$.
\end{proof}

\begin{figure}[h]
\input{Figure0}
\caption{Illustration of Lemma \ref{lem: main lemma}}
\end{figure}
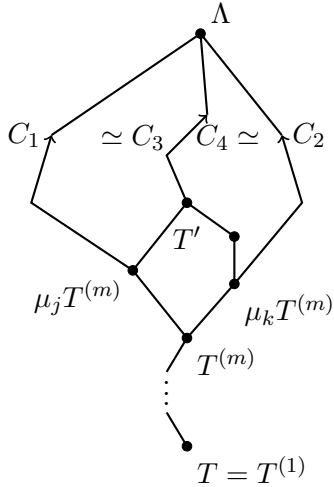

\begin{thm}[Garver-McConville \cite{GM}]\label{thm: no gap for finite type}
The No Gap Conjecture holds for all simply laced cluster tilted algebras of finite type.
\end{thm}

\begin{proof}
Maximal green sequences of a cluster tilted algebra $\Lambda$ are elements of $\cM G(\Lambda[1],\Lambda)$. If $\Lambda$ has finite representation type, all elements are polygonally deformable into each other by Lemma \ref{lem: main lemma}. In the simply laced case, all elementary polygonal deformations change the length of a maximal green sequence by at most one. So, the set of lengths will have no gaps.
\end{proof}

\begin{prop}\label{prop: finite comparable and MGS}
Suppose that $T$ has the property that there are only finitely many $T'$ which are comparable to $T$ ($T'\ge T$ or $T'\le T$). Then we have the following.

	\begin{enumerate}
	\item There exists a maximal green sequence going through $T$.
	\item There are only finitely many maximal green sequences going through $T$.
	\item Any two maximal green sequences going through $T$ lie in the same polygonal deformation class.
	\end{enumerate}

\end{prop}

\begin{proof}
The set of maximal green sequences which go through $T$ is in bijection with the product $\cM G(\Lambda[1],T)\times \cM G(T,\Lambda)$ which is finite and nonempty. Given two, $C_1=(C_1',C_1'')$ and $C_2=(C_2',C_2'')$, polygonal deformations $C_1'\sim C_2'$ and $C_1''\sim C_2''$ give a polygonal deformation $C_1\sim C_2$.
\end{proof}

\subsection{Regular cluster tilting objects} We now specialize to tame algebras $\Lambda$ with unique null root $\eta$. We will show that every maximal green sequence for $\Lambda$ passes through some regular object and any two sequences which pass through the same object are polygonally equivalent.

Let $\cP$ denote the set of all (indecomposable) preprojective objects of $\rmod\Lambda$ and let $\cR$ denote the set of all rigid indecomposable regular objects. We consider $\cP,\cR$ to be subsets of the set of exceptional objects of $\cC_\Lambda$, the cluster category of $\Lambda$. Let $\cJ$ be the set of all other exceptional objects of $\cC_\Lambda$. Thus, $\cJ$ consists of the indecomposable preinjective and shifted projective objects $P_i[1]$.

\begin{defn}
A cluster tilting object will be called \emph{regular} if it contains at least one component in $\cP$ and one component in $\cJ$.
\end{defn}

\begin{lem}
A cluster tilting object $T$ for tame $\Lambda$ has at most $n-2$ regular components.
\end{lem}

\begin{proof}
We know that, for any almost complete cluster tilting object $T=T_1\oplus \cdots\oplus T_{n-1}$, there is a real Schur roots $\beta$ so that $\brk{\undim T_i,\beta}=0$ for all $i$. If all $T_i$ were regular, then $\beta=\eta$, the null root, is (up to multiplication by a scalar) the unique solution of the linear equation $\brk{\undim T_i,\eta}=0$ for all $i$. This is impossible since $\eta$ is not a real root.
\end{proof}

Since every maximal green sequence starts with all components in $\cJ$ and ends with all components in $\cP$, there is an object $T$ in the sequence which contains at least one of each. 

\begin{prop}\label{prop: every MGS goes through a regular T}
Every maximal green sequence passes through a regular cluster.\qed
\end{prop}

\subsection{Maximal green sequences through regular objects}

In this subsection we will show that every regular cluster tilting object $T$ lies in some maximal green sequence and that any two such sequences are polygonally equivalent. To prove this, we need to recall some constructions from \cite{BHIT}.

By Dlab-Ringel \cite{DR}, there is a positive integer $m$ so that for any preprojective root $\alpha$ or preinjective root $\beta$ we have
	\[
	\tau^{-m}\alpha=\alpha+\delta(\alpha)\eta \quad \text{ and } \quad \tau^m\beta=\beta+\delta(\beta)\eta
	\]
where $\delta(\gamma)$ is the \emph{defect} of $\gamma$, $\eta$ is the null root, and $\tau=-E^{-1}E^t$ is Auslander-Reiten translation.

Let $\cP_m$ be the set of all preprojective roots $\alpha$ with the property that $\tau^k\alpha$ is a projective root for some $0\le k<m$. Let $\cI_m$ be the set of all preinjective roots $\beta$ with the property that $\tau^{-k}\beta$ is injective for some $0\le k<m$. Recall from \cite{BHIT} the following subsets of $\RR^n$:
	\[
	\cV_m:=\{x\in\RR^n: \brk{x,\beta} \ge0 \text{ for all }\beta\in\cI_m\}
	\]
	\[
	\cW_m:=\{x\in\RR^n:  \brk{x,\alpha} >0 \text{ for some }\alpha\in\cP_m\}.
	\]
Let $H(\eta)$ be the hyperplane in $\RR^n$ given by the equation $\brk{x,\eta}=0$ and let
	\[
	D(\eta)=\{x\in H(\eta): \brk{x,\alpha} \le0 \text{ for all preprojective }\alpha\}.
	\]
For any cluster tilting object $T=T_1\oplus\cdots\oplus T_n$ let $R(T)\subseteq \RR^n$ be the set
	\[
	R(T)=\left\{\sum a_i\undim T_i: a_i\ge0\right\}.
	\]

\begin{rem}\label{rem: D(eta) is convex}
Note that $D(\eta), \cV_m$ and $\RR^n\backslash\cW_m$ are convex. Also, $\eta\in D(\eta)$ and $-\eta\in H(\eta)\setminus D(\eta)$ since
$
	\brk{-\eta,\undim P}=\brk{\undim P,\eta} > 0
$
for all projective modules $P$. Furthermore, $D(\eta)$ contains all regular roots $\rho=\undim R$ where $R\in\cR$ since these are exactly the roots which lie on $H(\eta)$ and satisfy $\Hom_\Lambda(R,M)=0$ for all preprojective $M$ (making $\brk{\rho,\alpha}\le0$ for $\alpha$ preprojective).
\end{rem}

We recall the following properties of these regions.

\begin{prop}[See \cite{BHIT}, Section 4]\label{prop: regions}
	
	\begin{enumerate}
	\item $\cV_m\cap H(\eta)=D(\eta)$
	\item $\cW_m\cap H(\eta)=H(\eta)\setminus D(\eta)$.
	\item For all $T$, the interior of $R(T)$ lies either entirely in $\cV_m$ or in its complement.
	\item For all $T$, the interior of $R(T)$ lies either entirely in $\cW_m$ or in its complement.
	\item If $R(T)$ lies in $\cV_m$ (or $\cW_m$) and $R(\mu_kT)$ lies outside $\cV_m$ (or $\cW_m$) then $\mu_k$ is a green mutation (in both cases).
	\item $\cV_m\setminus \cW_m$ contains $R(T)$ for all but finitely many cluster tilting objects $T$.
	\item For all $T$, the interior of $R(T)$ is disjoint from $D(\eta)$.
	\item $\tau D(\eta)=D(\eta)$.
	\end{enumerate}
	
\end{prop}

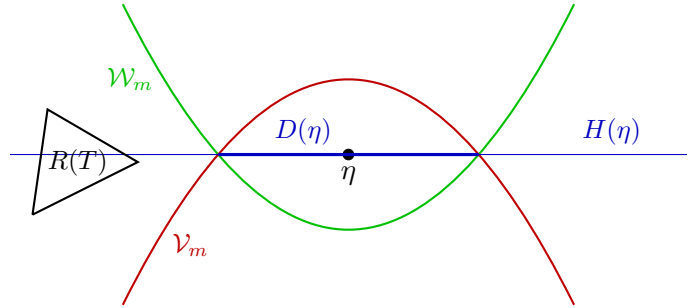
\begin{figure}[h]
\input{Figure1}
\caption{Illustration of Proposition \ref{prop: regions}}
\end{figure}

\begin{lem}\label{lem: regular cones}
If $T$ is a regular cluster tilting object then the interior of $R(T)$ meets $H(\eta)$ and therefore lies in $\cW_m$ and is disjoint from $\cV_m$. In particular, $R(T)$ is on the red side of $\partial \cV_m$ and on the green side of $\partial \cW_m$. Furthermore, there are at most finitely many such $T$.
\end{lem}

\begin{proof}
Since $T$ contains a preprojective and preinjective (or shifted projective) summand, $R(T)$ has points on both sides of the hyperplane $H(\eta)$. So, $H(\eta)$ meets the interior of $R(T)$. By (7) in the proposition above, the interior of $R(T)$ meets $H(\eta)\setminus D(\eta)$. By (1) and (3), this implies that the interior of $R(T)$ is disjoint from $\cV_m$. By (2) and (4), the interior of $R(T)$ is contained in $\cW_m$. By (6) there are only finitely many such $T$.
\end{proof}

\begin{lem}
For any regular cluster $T$ there are only finitely many $T'$ comparable with $T$.
\end{lem}

\begin{proof}
Suppose that $T'\ge T$. Then there is a green mutation from $T$ to $T'$. Since $T$ is on the red side of $\partial\cV_m$, $T'$ is also on the red side of $\partial\cV_m$ and therefore in the complement of $\cV_m$. By (6) in the above proposition there are only finitely many such $T'$.

Similarly, $T$ is on the green side of $\partial \cW_m$ and therefore so is any $T''\le T$. So, the interior of $R(T'')$ lies in $\cW_m$. Again there are only finitely many such $T''$ by (6) above.
\end{proof}

By Proposition \ref{prop: finite comparable and MGS} this implies the following.

\begin{prop}\label{prop: all MGS though T are equivalent}
Suppose $T$ is a regular cluster tilting object.
\begin{enumerate}
\item There exists a maximal green sequence going through $T$.
\item There are only finitely many maximal green sequences going through $T$.
\item Any two maximal green sequences going through $T$ lie in the same polygonal deformation class.\qed
\end{enumerate}
\end{prop}

\begin{rem}
We also note that there is at least one regular cluster tilting object, namely the direct sum of $P_i[1]$ with all $P_j$ for $i\neq j$ where $i$ is a source in the quiver of $\Lambda$.
\end{rem}

%%%%%%%%%%%%%%%%%%%%%%%%%%%%%%%%%%%%%%%%%%
%%%
%%%   The Geometry of Regular Objects
%%%
%%%%%%%%%%%%%%%%%%%%%%%%%%%%%%%%%%%%%%%%%%

\section{The Geometry of Regular Objects}

The set $\Reg(\Lambda)$ of regular cluster tilting objects of $\Lambda$ forms a poset where $T\le T'$ if there is a sequence of green mutations taking $T$ to $T'$. When this happens, there is a maximal green sequence containing both $T$ and $T'$. Using this as an intermediate step, we see that any maximal green sequence going through $T$ is polygonally equivalent to any maximal green sequence passing through $T'$. Thus it remains to show the following:

\begin{prop}\label{prop: Reg is connected}
The poset $\Reg(\Lambda)$ of regular cluster tilting objects of a tame hereditary algebra $\Lambda$ is connected.
\end{prop}

\subsection{The regular cluster fan} We introduce the following subspaces of $\RR^n$.

\begin{defn}
The \emph{regular cluster fan} $X(\Lambda)$ is defined to be the union of the cones $R(T)$ where $T$ is a regular cluster, and $Y(\Lambda)$ is the intersection of $X(\Lambda)$ with the hyperplane $H(\eta)$.
\end{defn}

To prove Proposition \ref{prop: Reg is connected}, we will show that $Y(\Lambda)$ is contractible. This will produce a connected subgraph of $\Reg(\Lambda)$ containing all of its vertices implying that the entire graph $\Reg(\Lambda)$ is connected.

\begin{lem}\label{lem: transversality}
The space $Y(\Lambda)$ is closed and decomposes into a union of polytopes $P(T):=R(T)\cap H(\eta)$ where $T\in\Reg(\Lambda)$. Moreover, the dual graph to this decomposition is a subgraph of $\Reg(\Lambda)$ containing all of its vertices, i.e., two polytopes $P(T)$, $P(T')$ share a face if and only if $T$, $T'$ differ by a mutation.
\end{lem}

\begin{proof}
The fact that $Y(\Lambda)$ is closed follows from Lemma \ref{lem: regular cones} as $X(\Lambda)$ is a finite union of closed subspaces. 

To show that $Y(\Lambda)$ decomposes into polytopes, we observe that every face of every $R(T)$ which meets $H(\eta)\setminus D(\eta)$ crosses $H(\eta)$ transversely. This is because, in order for a linear simplex to not be transverse to a hyperplane, it must be contained in the hyperplane. But, if the vertices of a face of $R(T)$ lie in $H(\eta)$, they must be regular roots. By Remark \ref{rem: D(eta) is convex} all regular roots lie in $D(\eta)$. Since $D(\eta)$ is convex, this would imply the entire face lies in $D(\eta)$, contradicting the assumption. 

Consequently, the intersection of two polytopes $P(T)$, $P(T')$ is the intersection of $H(\eta)\setminus D(\eta)$ with the common codimension one face of $R(T)$ and $R(T')$ for two regular clusters $T$, $T'$. Thus, $T$ and $T'$ differ by a mutation. Therefore, the dual graph of this decomposition is a subgraph of $\Reg(\Lambda)$ which by Lemma \ref{lem: regular cones} contains all of its vertices.
\end{proof}

\begin{eg}
This example illustrates the space $Y(\Lambda)$. Consider the quiver $\footnotesize{\xymatrix@C=12pt{1 & 2\ar[l] & 3\ar[l] & 4\ar[l]\ar@/_/[lll]}}$ of type $\widetilde{A}_{3,1}$. Recall (cf. e.g., \cite{IOTW2}) any two adjacent cones $R(T)$, $R(T')$ intersect along a semi-invariant domain $D(\beta)$ for some root $\beta$. By the previous Lemma, the domains $D(\beta)$ cut $Y(\Lambda)$ into triangles and squares as illustrated in Figure \ref{figure 2} (the large outside region is a square containing the point at infinity); the dotted lines indicate the boundary of $Y(\Lambda)$.

The domains $D(\beta)$ intersect the boundary at (dimension vectors of) regular representations, as illustrated in the figure. Here $R$ is the regular representation with dimension vector $(1,1,0,1)$.

	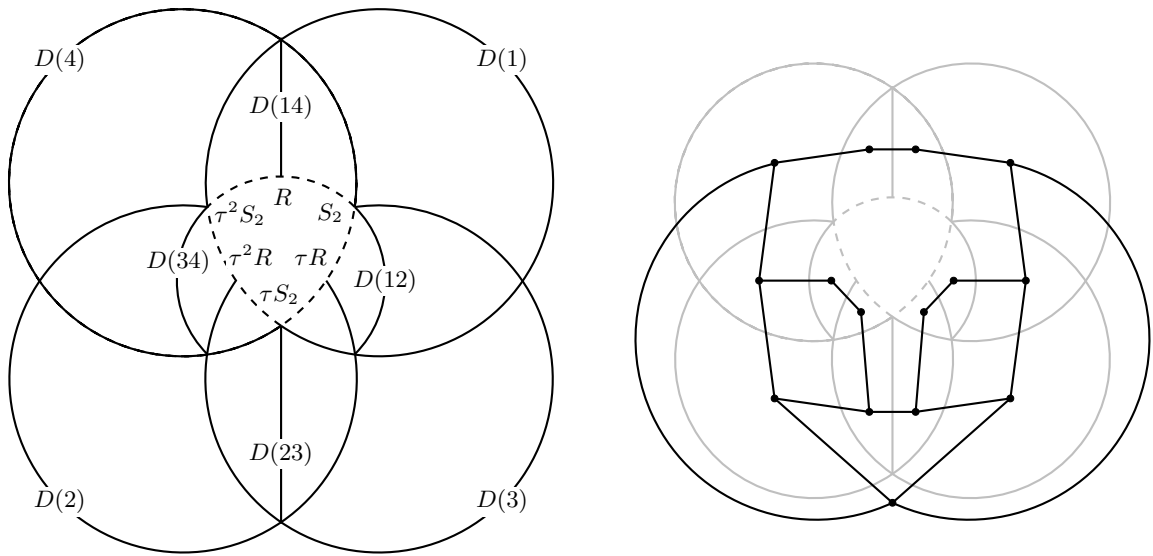
\begin{figure}[h]

	\input{Figure2}
	\caption{The space $Y(\Lambda)$ and its dual graph in type $\widetilde{A}_{3,1}$}
	\label{figure 2}

	\end{figure}

\end{eg}

Recall that a subspace $X\subseteq\RR^n$ is said to be \emph{star-shaped} around $x_0\in X$ if for every point $x\in X$ the line $\{tx_0+(1-t)x$: $0\le t\le 1\}$ is contained in $X$. Clearly, a star-shaped space is contractible. We recall without proof the following elementary properties of star-shaped regions.

\begin{lem}\label{lem: star shaped}
	
	\begin{enumerate}
	\item If $X$ is star-shaped about $x_0$, then so is its closure $\overline{X}$.
	\item If $C$ is a closed convex cone containing a point $x_0$, then the compliment $\RR^n\setminus C$ is star-shaped around $x_0$.
	\end{enumerate}
	
\end{lem}

In particular, (2) implies that $H(\eta)\setminus D(\eta)$ is star-shaped around $-\eta$, as $D(\eta)$ is convex and contains $\eta$. As star-shaped regions are contractible, the proof of Proposition \ref{prop: Reg is connected} reduces to the following.

\begin{prop}\label{prop: X cap H is star-shaped}
The space $Y(\Lambda)$ is the closure of the star-shaped set $H(\eta)\setminus D(\eta)$, and hence is star-shaped.
\end{prop}

Proposition \ref{prop: Reg is connected} is an immediate corollary of the above proposition. Indeed, Proposition \ref{prop: X cap H is star-shaped} implies the space $Y(\Lambda)$ has only one component, hence it's dual graph $\Gamma$ is connected. By Lemma \ref{lem: transversality}, $\Gamma$ is a subgraph of the (underlying graph of) $\Reg(\Lambda)$, hence $\Reg(\Lambda)$ is also connected.

\subsection{Some lemmas}

For the proof of Proposition \ref{prop: X cap H is star-shaped} we require a few technical lemmas.

\begin{lem}\label{lem: point with negative coordinates}
For any $x\in\RR^n$ there exists $J\subseteq \{j\,:\,x_j<0\}$ and $a_j>0$ for all $j\in J$ so that

	\begin{enumerate}
	\item $y:=x+\sum_{j\in J} a_j\undim P_j$ has nonnegative coordinates.
	\item $y_j=0$ for all $j\in J$.
	\end{enumerate}
	
\end{lem}

\begin{proof}
The proof is by induction on the number of negative coordinates of $x$. If this number is zero, then $J=\emptyset$ and $y=x$ satisfies the required conditions. Otherwise, choose the largest index $k$ so that $x_k<0$ where the vertices are ordered so that every projective $P_i$ has support at vertices $j\le i$. Let $a_k=|x_k|$ and let $x'=x+a_k\undim P_k$. Then

	\begin{enumerate}
	\item $x_k'=0$,
	\item $x'$ has fewer negative coordinates than $x$,
	\item if $x'_i<0$ then $x_i<0$.
	\end{enumerate}

By induction, there exists $J'\subseteq\{j\,:\,x_j'<0\}\subseteq \{j\,:\,x_j<0\}$ and $a_j$, $j\in J'$, so that
\[
	y:=x'+\sum_{j\in J'}a_j\undim P_j=x+a_k\undim P_k+\sum_{j\in J'}a_j\undim P_j
\]
has the desired properties.
\end{proof}

\begin{lem}\label{lem: point in H(eta) with neg coord are in a reg cluster}
Suppose $x\in H(\eta)\setminus D(\eta)$ has at least one nonpositive coordinate. Then $x\in R(T)$ for some regular cluster tilting object $T$.
\end{lem}

\begin{proof}
When $x$ has at least one negative coordinate, we have by Lemma \ref{lem: point with negative coordinates} that $x=y-\sum_{j\in J} a_j\undim P_j$ where $y_j=0$ for all $j\in J$ and $y_i\ge0$ for all $i$. When all coordinates of $x$ are $\ge0$, let $y=x$ and let $J$ be the set of all $j$ so that $y_j=x_j=0$ and let $a_j=0$ for all $j$.

Since $J$ is nonempty in both cases, $y$ has support in a proper subquiver $Q'$ of $Q$, the quiver of $\Lambda$, with vertex set $Q'_0=Q_0\setminus J$. Since $\Lambda$ is tame, $Q'$ has finite representation type. Therefore, by the Generic Decomposition Theorem (see \cite{DW}, \cite{S}), $y$ can be written as $y=\sum_{i\in Q_0'}b_i\undim M_i$ where $b_i\ge0$ and $M_i$ are indecomposable representations of the modulated quiver $Q'$ which do not extend each other. Then, $T=\bigoplus M_i\oplus \bigoplus P_j[1]$ is a cluster tilting object for $\Lambda$ and $x=\sum b_i \undim M_i+\sum a_j\undim P_j[1]\in R(T)$.

To show $T$ is regular, suppose that it is not. Then all $M_i$ are preinjective or regular. So,
	\[
	\brk{x,\eta}=\sum b_i\brk{\undim M_i,\eta}-\sum a_j\brk{\undim P_j,\eta}\le -\sum a_j\brk{\undim P_j,\eta}
	\]
which is negative, contradicting the assumption that $x\in H(\eta)$ (as $\brk{x,\eta}=0$) in the case when $a_j>0$. Thus, $\brk{x,\eta}=0$ only in the case when $a_j=0$ for all $j\in J$ and, furthermore, each $M_i$ must be regular since $\brk{\undim M_i,\eta}<0$ for preinjective $M_i$. But then $x$ is a positive linear combination of regular roots making it an element of the convex set $D(\eta)$ contrary to assumption by Remark \ref{rem: D(eta) is convex}. So, $T$ is regular in all cases.
\end{proof}

\begin{lem}\label{lem: claim 1}
If $x\in H(\eta)\setminus D(\eta)$, there is an integer $k\ge 0$ so that $\tau^{-k}x$ has a negative coordinate.
\end{lem}

\begin{proof}
Since $x\notin D(\eta)=H(\eta)\cap \cV_m$, $x\notin \cV_m$. Therefore, there exists some preinjective $\beta\in\cI_m$ so that $\brk{x,\beta}<0$. By definition of $\cI_m$, there is some $0\le k<m$ so that $\tau^{-k}\beta$ is the dimension vector of some injective module $I_i$. Then
	\[
	\brk{x,\beta}=\brk{\tau^{-k}x,\undim I_i}=f_i(\tau^{-k}x)_i<0
	\]
where $f_i=\dim\End_\Lambda(I_i)>0$. So, $(\tau^{-k}x)_i<0$, proving the lemma.
\end{proof}

%%%%%%%%%%%%%%%%%%%%%%%%%%%%%%%%%%%%%%%%%%
%%%
%%%   Proof of the Main Theorem
%%%
%%%%%%%%%%%%%%%%%%%%%%%%%%%%%%%%%%%%%%%%%%

\section{Proof of the Main Theorem}

\begin{proof}[Proof of Proposition \ref{prop: X cap H is star-shaped}.]

By Lemma \ref{lem: regular cones}, the interior of $R(T)$ for any regular $T$ is disjoint from $\cV_m\setminus \cW_m$ and therefore from $D(\eta)=(\cV_m\setminus \cW_m)\cap H(\eta)$. So, $R(T)\cap H(\eta)$ is contained in the closure of $H(\eta)\setminus D(\eta)$ and the same holds for $Y(\Lambda)=\bigcup_{T\in\Reg(\Lambda)} P(T)$ (where $P(T)=R(T)\cap H(\eta)$). So, it suffices to show that $H(\eta)\setminus D(\eta)$ is contained in the closed set $X(\Lambda)$. Let $x\in H(\eta)\setminus D(\eta)$.

By Lemma \ref{lem: claim 1} there is some $k\ge 0$ so that $\tau^{-k}x$ has a coordinate $\le0$. Take the minimal such $k$. Since $D(\eta)$ and $H(\eta)$ are invariant under $\tau$, $\tau^{-k}x\in H(\eta)\setminus D(\eta)$. By Lemma \ref{lem: point in H(eta) with neg coord are in a reg cluster}, $\tau^{-k}x\in R(T)$ for some regular cluster tilting object $T$. If $k=0$ we are done. So, suppose $k>0$. We claim that $\tau^kT$ is a regular cluster and $x\in R(\tau^kT)$. We use the fact that $\tau R(T)=R(\tau T)$ if $T$ has no projective summands.\\
We claim that $\tau^jT$ has no projective summands for $0\le j<k$. Suppose the claim is false, and take $j\ge0$ minimal so that $\tau^jT$ has a projective summand $P_i$. Then, for $0\le p=k-j-1<k$, $\tau^{-p-1}x \in R(\tau^jT)$. This implies that $x$ is a nonnegative linear combination of ext-orthogonal roots including $\undim P_i$. So, $\brk{\tau^{-p-1}x,\undim P_i}\ge0$. Since $\tau$ is an isometry and $\tau P_i=I_i[1]$, we have
	\[
	\brk{\tau^{-p}x,\undim \tau P_i}=\brk{\tau^{-p}x,-\undim I_i}=-f_i(\tau^{-p}x)_i\ge0
	\]
which implies that $(\tau^{-p}x)_i\le0$. This contradicts the minimality of $k$ since $p<k$. Thus $\tau^kT$ is a regular cluster and $x\in R(\tau^kT)$.
\end{proof}

\begin{thm}\label{thm: polygon theorem}
For any tame algebra $\Lambda$, any two maximal green sequences can be deformed into each other by polygonal moves.
\end{thm}

\begin{proof}
By Proposition \ref{prop: every MGS goes through a regular T} every maximal green sequence goes through a regular cluster $T$. Any two maximal green sequences which go through the same regular $T$ are polygonally equivalent by Proposition \ref{prop: all MGS though T are equivalent}. By Proposition \ref{prop: Reg is connected} the graph $\Reg(\Lambda)$ of all regular clusters is connected. This implies that any maximal green sequence going though one regular cluster $T$ will be polygonally equivalent to any maximal green sequence going through any other regular cluster $T'$. Therefore, any two maximal green sequences for $\Lambda$ are polygonally equivalent.
\end{proof}

\begin{cor}
The No Gap Conjecture holds for any simply laced quiver of tame type.
\end{cor}

As observed in \cite{BDP} one useful application of the No Gap Conjecture is the ability to compute the length of the longest maximal green sequence for a tame algebra. Recall from \cite{BDP} that the \emph{empirical maximum length} for maximal green sequences for $\Lambda$ is the smallest positive integer $\ell_0$ so that the number of maximal green sequences of length $\le \ell_0$ is positive and equal to the number of maximal green sequences of length $\le \ell_0+1$. When the No Gap Conjecture holds this number must be equal to the length of the longest maximal green sequence since the existence of a longer maximal green sequence would produce a gap, namely $\ell_0+1$, in the sequence of lengths of maximal green sequences. In particular, for the quiver $\widetilde{A}_{n,1}$, with $n$ arrows going clockwise and one arrow going counterclockwise for $n\le7$, the maximum length of a maximal green sequence is equal to $n(n+3)/2$ by a calculation of \cite{BDP}. This has been shown directly for all $n$ in \cite{Kase}.

\section*{Acknowledgements}

The authors would like to thank Thomas Br\"ustle and Gordana Todorov for many long discussions on the subject of maximal green sequences. This paper relies heavily on our previous joint work \cite{BHIT}. Additionally, Gordana Todorov gave the authors helpful guidance on the proof of the key Proposition \ref{prop: X cap H is star-shaped}. Ying Zhou and Hugh Thomas also contributed to this paper by indirectly giving the authors ideas on how to shorten the proof. In particular, Hugh Thomas explained in his lecture at Sherbrooke in September, 2015 how the Cambrian lattice is a lattice quotient of the weak Bruhat order. However, the crucial idea for this paper comes from Alexander Garver and Thomas McConville \cite{GM} and the authors thank Al Garver in particular for explaining his ideas about the No Gap Conjecture.

%%%%%%%%%%%%%%%%%%%%%%%%%%%%%%%%%%%%%%%%%%
%%%
%%%   Bibliography
%%%
%%%%%%%%%%%%%%%%%%%%%%%%%%%%%%%%%%%%%%%%%

\bibliographystyle{amsplain}
\bibliography{BibliographyBibTeX}

%\begin{thebibliography}{aa}
%\input{Bibliography}
%\end{thebibliography}

\end{document}

%% file: AuthorInformation.tex
\author{Stephen Hermes}
\address{Department of Mathematics, Wellesley College, Wellesley, MA 02481}\email{shermes@wellesley.edu}

\author{Kiyoshi Igusa}
\address{Department of Mathematics, Brandeis University, Waltham, MA 02454}\email{igusa@brandeis.edu}

\subjclass[2010]{
16G20; 20F55}

\keywords{cluster category, cluster tilting objects, maximal green sequences, valued quivers, cluster mutation}

%% file: Figure0.tex
% Figure for proof of Theorem 1,2,3
\begin{center}
\begin{tikzpicture}[scale=.9]%[>=Stealth]
%\draw[help lines, color=blue] (0,0) grid (8,6);
%\foreach \x in {0,2,...,8}\draw (\x,0) node{\x};
%\foreach \y in {0,2,...,6}\draw (0,\y) node{\y};
%\draw[fill] (0,0) circle[radius=2pt]node[right]{$T^1$};
\draw[fill] (0.3,0.9) circle[radius=2pt]node[anchor=north west]{$T=T^{(1)}$};
\draw[fill] (0.3,2.5) circle[radius=2pt]node[anchor=north west]{$T^{(m)}$};
\draw[thick] (0,1.8) node{$\vdots$};
\draw[thick] (0.3,0.9)--(0,1.4) (0,2)--(.3,2.5);
\draw[fill] (1,3.3) circle[radius=2pt]node[anchor=north west]{$\mu_kT^{(m)}$};
\draw[fill] (1,4) circle[radius=2pt];
\draw[fill] (.3,4.5) circle[radius=2pt] (.3,4.3)node[below]{$T'$};
\draw[fill] (-.5,3.5) circle[radius=2pt]node[anchor=north east]{$\mu_jT^{(m)}$};
\draw[thick] (0.3,2.5)--(-.5,3.5)--(.3,4.5)--(1,4)--(1,3.3)--cycle;
\draw[fill] (.5,7) circle[radius=2pt] node[anchor=south west]{$\Lambda$};
\draw[thick,->] (-.5,3.5)-- (-2, 4.5) -- (-1.7,5.5) node[left]{$C_1$};
\draw[thick] (-1.7,5.5) --(.5,7); 
\draw[thick,->] (1,3.3)-- (2, 4.5) -- (1.7,5.5) node[right]{$C_2$};
\draw[thick] (1.7,5.5) --(.5,7); 
\draw[thick,->] (.3,4.5)-- (0, 5.2)--(.6,5.8);
\draw[thick] (.9,5.8)  node[below]{$C_4\simeq$} (-.5,5.8)node[below]{$\simeq C_3$};
\draw[thick](.6,5.8) -- (.5,7);
\end{tikzpicture}
\end{center}

%% file: Figure1.tex
\begin{center}
\begin{tikzpicture}%[scale=3]

%\draw[help lines=1,thick] (-5,-5) grid (5,4);
%\draw[help lines=.2,thin] (-5,-5) grid (5,4);
\draw[thick,color=\green] (-3,2)..controls (-1,-2) and (1,-2)..(3,2);
\draw[thick,color=\red] (-3,-2)..controls (-1,2) and (1,2)..(3,-2);
%\foreach \x in {-8,-6,...,8}\draw (\x,0) node{\x};\foreach \y in {-6,-4,...,8}\draw (0,\y) node{\y};
%\draw[thick,color=blue] (0,1) ellipse [x radius=2.8cm,y radius=2.1cm];
\begin{scope}%[xshift=-4cm,yshift=10cm]
	\draw[fill] (0,0) circle[radius=2pt] (0,-.02)node[below]{$\eta$};
	\draw[very thick,color=\blue] (-1.73,0)--(1.73,0);
	\draw[color=\blue] (-.6,0) node[above]{\small$D(\eta)$};
	\draw[color=\blue] (3.5,0) node[above]{\small$H(\eta)$};
\end{scope}
\draw[color=\blue] (-4.5,0)--(4.5,0);
\begin{scope}
	\draw[thick] (-4.2,-.8)--(-2.8,-.1)--(-4,.6)--(-4.2,-.8);
	\draw (-3.6,-.1) node{\small$R(T)$};
\end{scope}
\draw[color=\red] (-2.1,-1.2)node{\small$\cV_m$};
\draw[color=\green] (-2.9,1)node{\small$\cW_m$};

\end{tikzpicture}
\end{center}

%% file: Figure2.tex
	\begin{center}
	\newdimen\dd
	\dd=6em

	\begin{tikzpicture}[scale=0.8]

\begin{scope}[rotate=-45] %main pic    	 
        
    \begin{scope} % simple walls
    
    % wall 1      
    \draw[wall] (315:0.325\dd) arc (-79.5:233:1.25\dd) ; 
    \draw[nullwall] (315:0.325\dd) arc (360-79.5:233:1.25\dd) ;
     
	% wall 2
    \draw[wall] (315+90:0.325\dd) arc (79.5:-233:1.25\dd) ; 
%   \draw[nullwall] (315+90:0.325\dd) arc (360+79.5:-233:1.25\dd) ;

	% wall 3
    \draw[wall] (315-90:0.325\dd) arc (-79.5-90:233-90:1.25\dd) ; 
%   \draw[nullwall] (315-90:0.325\dd) arc (360-79.5-90:233-90:1.25\dd) ;         

	% wall 4
    \draw[wall] (315:0.325\dd) arc (79.5-90:-233-90:1.25\dd) ; 
    \draw[nullwall] (315:0.325\dd) arc (360+79.5-90:-233-90:1.25\dd) ;
    
	\end{scope}

	\begin{scope} % central arc
                
	\draw[wall] (0:0.75\dd) arc (0:90:0.75\dd) ;
	\draw[nullwall] (90:0.75\dd) arc (90:180:0.75\dd) ;
	\draw[wall] (180:0.75\dd) arc (180:270:0.75\dd) ;  
	
	\end{scope}
	
	\begin{scope} % short walls
	
	\draw[wall] (315:0.325\dd) -- (315:1.75\dd) ;
	\draw[wall] (315-180:0.75\dd) -- (315-180:1.75\dd) ;
	
	\end{scope}	
	
	\begin{scope} % domain labels
		\node[ann] at (0:2.25\dd) {$D(3)$} ;
		\node[ann] at (90:2.25\dd) {$D(1)$} ;
		\node[ann] at (180:2.25\dd) {$D(4)$} ;
		\node[ann] at (270:2.25\dd) {$D(2)$} ;
		
		\node[ann] at (315:1.25\dd) {$D(23)$} ;
		\node[ann] at (135:1.25\dd) {$D(14)$} ;
		
		\node[ann] at (45:0.75\dd) {$D(12)$} ;
		\node[ann] at (215:0.75\dd) {$D(34)$} ;
	\end{scope}
	
	\begin{scope}[shift=(135:0.3\dd),rotate=-17] % regular root labels
		\node[ann] at (30:0.25\dd) {$\tau R$} ;
		\node[ann] at (90:0.4\dd) {$S_2$} ;
		\node[ann] at (150:0.3\dd) {$R$} ;
		\node[ann] at (210:0.35\dd) {$\tau^2S_2$} ;
		\node[ann] at (270:0.25\dd) {$\tau^2 R$} ;
		\node[ann] at (330:0.4\dd) {$\tau S_2$} ;
	\end{scope}
	
\end{scope}

\begin{scope}[xshift=4.4\dd,rotate=-45,scale=0.8] % copy for graph
     
    \begin{scope} % simple walls
    
    % wall 1      
    \draw[faintwall] (315:0.325\dd) arc (-79.5:233:1.25\dd) ; 
    \draw[faintnullwall] (315:0.325\dd) arc (360-79.5:233:1.25\dd) ;
     
	% wall 2
    \draw[faintwall] (315+90:0.325\dd) arc (79.5:-233:1.25\dd) ; 
%   \draw[nullwall] (315+90:0.325\dd) arc (360+79.5:-233:1.25\dd) ;

	% wall 3
    \draw[faintwall] (315-90:0.325\dd) arc (-79.5-90:233-90:1.25\dd) ; 
%   \draw[nullwall] (315-90:0.325\dd) arc (360-79.5-90:233-90:1.25\dd) ;         

	% wall 4
    \draw[faintwall] (315:0.325\dd) arc (79.5-90:-233-90:1.25\dd) ; 
    \draw[faintnullwall] (315:0.325\dd) arc (360+79.5-90:-233-90:1.25\dd) ;
    
	\end{scope}

	\begin{scope} % central arc
                
	\draw[faintwall] (0:0.75\dd) arc (0:90:0.75\dd) ;
	\draw[faintnullwall] (90:0.75\dd) arc (90:180:0.75\dd) ;
	\draw[faintwall] (180:0.75\dd) arc (180:270:0.75\dd) ;  
	
	\end{scope}
	
	\begin{scope} % short walls
	
	\draw[faintwall] (315:0.325\dd) -- (315:1.75\dd) ;
	\draw[faintwall] (315-180:0.75\dd) -- (315-180:1.75\dd) ;
	
	\end{scope}	
	
	\begin{scope} % dual graph vertices
    
    \draw[fill] (0:1.5\dd) circle[radius=2pt];
    \draw[fill] (90:1.5\dd) circle[radius=2pt];
    \draw[fill] (180:1.5\dd) circle[radius=2pt];
    \draw[fill] (270:1.5\dd) circle[radius=2pt];
    
    \draw[fill] (0:0.4\dd) circle[radius=2pt];
    \draw[fill] (270:0.4\dd) circle[radius=2pt];
    \draw[fill] (45:0.55\dd) circle[radius=2pt];
    \draw[fill] (225:0.55\dd) circle[radius=2pt];    
    
    \draw[fill] (45:1.2\dd) circle[radius=2pt];
    \draw[fill] (225:1.2\dd) circle[radius=2pt];
    
    \draw[fill] (125:1.2\dd) circle[radius=2pt];
    \draw[fill] (145:1.2\dd) circle[radius=2pt];  
    
    \draw[fill] (125:-1.2\dd) circle[radius=2pt];
    \draw[fill] (145:-1.2\dd) circle[radius=2pt];  
    
    \draw[fill] (315:2\dd) circle[radius=2pt];  

    \end{scope}
    
    \begin{scope} % dual graph edges
    
    \draw[wall] (0:1.5\dd) -- (45:1.2\dd) -- (90:1.5\dd) -- (125:1.2\dd) -- (145:1.2\dd) -- (180:1.5\dd) -- (225:1.2\dd) -- (270:1.5\dd) -- (125:-1.2\dd) -- (145:-1.2\dd) -- (0:1.5\dd) ;
    
    \draw[wall] (145:-1.2\dd) -- (0:0.4\dd) -- (45:0.55\dd) -- (45:1.2\dd) ; 
    
    \draw[wall] (125:-1.2\dd) -- (270:0.4\dd) -- (225:0.55\dd) -- (225:1.2\dd) ; 
    
    \draw[wall] (270:1.5\dd) -- (315:2\dd) -- (0:1.5\dd) ;
    \draw[wall] (-45:2\dd) arc (-70:122:1.625\dd) ;
    \draw[wall] (-45:2\dd) arc (-90+70:-90-122:1.625\dd) ;
 
    \end{scope}

\end{scope}

    \end{tikzpicture}
    
	\end{center}